\documentclass{conm-p-l}
\usepackage{amssymb}

\usepackage{amsfonts}
\usepackage{eucal}
\usepackage{upgreek}
\begin{document}

\font\eightrm=cmr8
\font\eightit=cmti8
\font\eighttt=cmtt8
\font\tensans=cmss10
\def\emp{\text{\tensans\O}}
\def\tci
{\hbox{\hskip1.8pt$\rightarrow$\hskip-11.5pt$^{^{C^\infty}}$\hskip-1.3pt}}
\def\nft
{\hbox{$n$\hskip3pt$\equiv$\hskip4pt$5$\hskip4.4pt$($mod\hskip2pt$3)$}}
\def\bbR{\mathrm{I\!R}}
\def\rto{\bbR\hskip-.5pt^2}
\def\rtr{\bbR\hskip-.7pt^3}
\def\rfo{\bbR\hskip-.7pt^4}
\def\rn{\bbR^{\hskip-.6ptn}}
\def\mr{\bbR^{\hskip-.6ptm}}
\def\bbZ{\mathsf{Z\hskip-4ptZ}}
\def\bbRP{\text{\bf R}\text{\rm P}}
\def\bbC{{\mathchoice {\setbox0=\hbox{$\displaystyle\rm C$}\hbox{\hbox
to0pt{\kern0.4\wd0\vrule height0.9\ht0\hss}\box0}}
{\setbox0=\hbox{$\textstyle\rm C$}\hbox{\hbox
to0pt{\kern0.4\wd0\vrule height0.9\ht0\hss}\box0}}
{\setbox0=\hbox{$\scriptstyle\rm C$}\hbox{\hbox
to0pt{\kern0.4\wd0\vrule height0.9\ht0\hss}\box0}}
{\setbox0=\hbox{$\scriptscriptstyle\rm C$}\hbox{\hbox
to0pt{\kern0.4\wd0\vrule height0.9\ht0\hss}\box0}}}}
\def\bbQ{{\mathchoice {\setbox0=\hbox{$\displaystyle\rm Q$}\hbox{\raise
0.15\ht0\hbox to0pt{\kern0.4\wd0\vrule height0.8\ht0\hss}\box0}}
{\setbox0=\hbox{$\textstyle\rm Q$}\hbox{\raise
0.15\ht0\hbox to0pt{\kern0.4\wd0\vrule height0.8\ht0\hss}\box0}}
{\setbox0=\hbox{$\scriptstyle\rm Q$}\hbox{\raise
0.15\ht0\hbox to0pt{\kern0.4\wd0\vrule height0.7\ht0\hss}\box0}}
{\setbox0=\hbox{$\scriptscriptstyle\rm Q$}\hbox{\raise
0.15\ht0\hbox to0pt{\kern0.4\wd0\vrule height0.7\ht0\hss}\box0}}}}
\def\bbQ{{\mathchoice {\setbox0=\hbox{$\displaystyle\rm Q$}\hbox{\raise
0.15\ht0\hbox to0pt{\kern0.4\wd0\vrule height0.8\ht0\hss}\box0}}
{\setbox0=\hbox{$\textstyle\rm Q$}\hbox{\raise
0.15\ht0\hbox to0pt{\kern0.4\wd0\vrule height0.8\ht0\hss}\box0}}
{\setbox0=\hbox{$\scriptstyle\rm Q$}\hbox{\raise
0.15\ht0\hbox to0pt{\kern0.4\wd0\vrule height0.7\ht0\hss}\box0}}
{\setbox0=\hbox{$\scriptscriptstyle\rm Q$}\hbox{\raise
0.15\ht0\hbox to0pt{\kern0.4\wd0\vrule height0.7\ht0\hss}\box0}}}}
\def\dimr{\dim_{\hskip.4pt\bbR\hskip-1.2pt}\w}
\def\dimc{\dim_{\hskip.4pt\bbC\hskip-1.2pt}\w}
\def\aff{\mathrm{A\hn f\hh f}\hs}
\def\cx{C\hskip-2pt_x\w}
\def\cy{C\hskip-2pt_y\w}
\def\cz{C\hskip-2pt_z\w}
\def\hyp{\hskip.5pt\vbox
{\hbox{\vrule width3ptheight0.5ptdepth0pt}\vskip2.2pt}\hskip.5pt}
\def\er{r}
\def\es{s}
\def\df{d\hskip-.8ptf}
\def\dz{\mathcal{D}}
\def\dzp{\dz^\perp}
\def\fv{\mathcal{F}}
\def\gr{\mathcal{G}}
\def\fvp{\fv_{\nrmh p}}
\def\wv{\mathcal{W}}
\def\vt{\mathcal{P}}
\def\tv{\mathcal{T}}
\def\vr{\mathcal{V}}
\def\xs{J}
\def\xl{S}
\def\cs{\mathcal{B}}
\def\zy{\mathcal{Z}}
\def\vtx{\vt_{\nh x}}
\def\fh{f}
\def\g{\mathtt{g}}
\def\rc{\theta}
\def\jm{\mathcal{I}}
\def\ke{\mathcal{K}}
\def\xc{\mathcal{X}_c}
\def\lz{\mathcal{L}}
\def\dla{\mathcal{D}_{\hskip-2ptL}^*}
\def\Lie{\pounds}
\def\lv{\Lie\hskip-1.2pt_v\w}
\def\lo{\lz_0}
\def\xe{\mathcal{E}}
\def\eo{\xe_0}
\def\lsq{\mathsf{[}}
\def\rsq{\mathsf{]}}
\def\hga{\hskip2.3pt\widehat{\hskip-2.3pt\gamma\hskip-2pt}\hskip2pt}
\def\hm{\hskip1.9pt\widehat{\hskip-1.9ptM\hskip-.2pt}\hskip.2pt}
\def\hg{\hskip.9pt\widehat{\hskip-.9pt\g\hskip-.9pt}\hskip.9pt}
\def\hna{\hskip.2pt\widehat{\hskip-.2pt\nabla\hskip-1.6pt}\hskip1.6pt}
\def\hdz{\hskip.9pt\widehat{\hskip-.9pt\dz\hskip-.9pt}\hskip.9pt}
\def\hdp{\hskip.9pt\widehat{\hskip-.9pt\dz\hskip-.9pt}\hskip.9pt^\perp}
\def\hmt{\hskip1.9pt\widehat{\hskip-1.9ptM\hskip-.5pt}_t}
\def\hmz{\hskip1.9pt\widehat{\hskip-1.9ptM\hskip-.5pt}_0}
\def\hmp{\hskip1.9pt\widehat{\hskip-1.9ptM\hskip-.5pt}_p}
\def\hk{\hskip1.5pt\widehat{\hskip-1.5ptK\hskip-.5pt}\hskip.5pt}
\def\hq{\hskip1.5pt\widehat{\hskip-1.5ptQ\hskip-.5pt}\hskip.5pt}
\def\txm{{T\hskip-2.9pt_x\w\hn M}}
\def\tyhm{{T\hskip-3.5pt_y\w\hm}}
\def\q{q}
\def\bq{\hat q}
\def\p{p}
\def\w{^{\phantom i}}
\def\x{u}
\def\y{v}
\def\vp{{\tau\hskip-4.55pt\iota\hskip.6pt}}
\def\evp{{\tau\hskip-3.55pt\iota\hskip.6pt}}
\def\vd{\vt\hh'}
\def\vdx{\vd{}\hskip-4.5pt_x}
\def\bz{b\hh}
\def\fe{F}
\def\fy{\phi}
\def\vl{\Lambda}
\def\hy{\mathcal{V}}
\def\vh{h}
\def\bc{C}
\def\mv{V}
\def\vo{V_{\nnh0}}
\def\ao{A_0}
\def\bo{B_0}
\def\uv{\mathcal{U}}
\def\sv{\mathcal{S}}
\def\svp{\sv_p}
\def\xv{\mathcal{X}}
\def\xvp{\xv_p}
\def\yv{\mathcal{Y}}
\def\yvp{\yv_p}
\def\zv{\mathcal{Z}}
\def\zvp{\zv_p}
\def\cv{\mathcal{C}}
\def\dy{\mathcal{D}}
\def\nv{\mathcal{N}}
\def\iv{\mathcal{I}}
\def\gkp{\Sigma}
\def\ret{\sigma}
\def\taw{\uptau}
\def\hs{\hskip.7pt}
\def\hh{\hskip.4pt}
\def\hn{\hskip-.4pt}
\def\nh{\hskip-.7pt}
\def\nnh{\hskip-1pt}
\def\hrz{^{\hskip.5pt\text{\rm hrz}}}
\def\vrt{^{\hskip.2pt\text{\rm vrt}}}
\def\vt{\varTheta}
\def\mtr{\Theta}
\def\op{\varTheta}
\def\vg{\varGamma}
\def\my{\mu}
\def\ny{\nu}
\def\gy{\lambda}
\def\lp{\lambda}
\def\ax{\alpha}
\def\lf{\widetilde{\lp}}
\def\bx{\beta}
\def\ay{a}
\def\by{b}
\def\gp{\mathrm{G}}
\def\hp{\mathrm{H}}
\def\kp{\mathrm{K}}
\def\gm{\gamma}
\def\Gm{\Gamma}
\def\Lm{\Lambda}
\def\Dt{\Delta}
\def\dg{\Delta}
\def\sj{\sigma}
\def\lg{\langle}
\def\rg{\rangle}
\def\lr{\langle\hh\cdot\hs,\hn\cdot\hh\rangle}
\def\vs{vector space}
\def\rvs{real vector space}
\def\vf{vector field}
\def\tf{tensor field}
\def\tvn{the vertical distribution}
\def\dn{distribution}
\def\pt{point}
\def\tc{tor\-sion\-free connection}
\def\ea{equi\-af\-fine}
\def\rt{Ric\-ci tensor}
\def\pde{partial differential equation}
\def\pf{projectively flat}
\def\pfs{projectively flat surface}
\def\pfc{projectively flat connection}
\def\pftc{projectively flat tor\-sion\-free connection}
\def\su{surface}
\def\sco{simply connected}
\def\psr{pseu\-\hbox{do\hs-}Riem\-ann\-i\-an}
\def\inv{-in\-var\-i\-ant}
\def\trinv{trans\-la\-tion\inv}
\def\feo{dif\-feo\-mor\-phism}
\def\feic{dif\-feo\-mor\-phic}
\def\feicly{dif\-feo\-mor\-phi\-cal\-ly}
\def\Feicly{Dif-feo\-mor\-phi\-cal\-ly}
\def\diml{-di\-men\-sion\-al}
\def\prl{-par\-al\-lel}
\def\skc{skew-sym\-met\-ric}
\def\sky{skew-sym\-me\-try}
\def\Sky{Skew-sym\-me\-try}
\def\dbly{-dif\-fer\-en\-ti\-a\-bly}
\def\cf{con\-for\-mal\-ly flat}
\def\ls{locally symmetric}
\def\ecs{essentially con\-for\-mal\-ly symmetric}
\def\rr{Ric\-ci-re\-cur\-rent}
\def\kf{Kil\=ling field}
\def\om{\omega}
\def\vol{\varOmega}
\def\og{\varOmega\hs}
\def\dv{\delta}
\def\ve{\varepsilon}
\def\zt{\zeta}
\def\kx{\kappa}
\def\mf{manifold}
\def\mfd{-man\-i\-fold}
\def\bmf{base manifold}
\def\bd{bundle}
\def\tbd{tangent bundle}
\def\ctb{cotangent bundle}
\def\bp{bundle projection}
\def\prc{pseu\-\hbox{do\hs-}Riem\-ann\-i\-an metric}
\def\prd{pseu\-\hbox{do\hs-}Riem\-ann\-i\-an manifold}
\def\Prd{pseu\-\hbox{do\hs-}Riem\-ann\-i\-an manifold}
\def\npd{null parallel distribution}
\def\pj{-pro\-ject\-a\-ble}
\def\pd{-pro\-ject\-ed}
\def\lcc{Le\-vi-Ci\-vi\-ta connection}
\def\vb{vector bundle}
\def\vbm{vec\-tor-bun\-dle morphism}
\def\kerd{\text{\rm Ker}\hskip2.7ptd}
\def\ro{\rho}
\def\sy{\sigma}
\def\ts{total space}
\def\pmb{\pi}
\def\pl{\partial}
\def\cro{\overline{\hskip-2pt\pl}}
\def\ddb{\pl\hskip1.7pt\cro}
\def\dbd{\cro\hskip-.3pt\pl}
\newtheorem{theorem}{Theorem}[section] 
\newtheorem{proposition}[theorem]{Proposition} 
\newtheorem{lemma}[theorem]{Lemma} 
\newtheorem{corollary}[theorem]{Corollary} 
  
\theoremstyle{definition} 
  
\newtheorem{defn}[theorem]{Definition} 
\newtheorem{notation}[theorem]{Notation} 
\newtheorem{example}[theorem]{Example} 
\newtheorem{conj}[theorem]{Conjecture} 
\newtheorem{prob}[theorem]{Problem} 
  
\theoremstyle{remark} 
  
\newtheorem{remark}[theorem]{Remark}

\title[Kil\-ling fields on compact pseu\-do\hs-\hn K\"ahler
manifolds]{Kil\-ling fields on compact pseu\-do\hs-\nh K\"ahler manifolds}
\author[A. Derdzinski]{Andrzej Derdzinski} 
\address{Department of Mathematics, The Ohio State University, 
Columbus, OH 43210, USA} 
\email{andrzej@math.ohio-state.edu} 
\author[I.\ Terek]{Ivo Terek} 
\address{Department of Mathematics, The Ohio State University, 
Columbus, OH 43210, USA} 
\email{terekcouto.1@osu.edu} 
\subjclass[2020]{Primary 53C50; Secondary 53C56}
\keywords{Compact pseu\-do-\hn K\"ah\-ler manifold, Kil\-ling vector field}
\def\leftmark{A.\ Derdzinski \&\ I.\ Terek}
\def\rightmark{Kil\-ling fields on pseu\-do\hs-\hn K\"ahler manifolds}

\begin{abstract}
We show that a Kil\-ling field on a compact pseu\-do\hs-\hn K\"ahler ddbar
manifold is necessarily (real) hol\-o\-mor\-phic. Our argument works without
the ddbar assumption in real dimension four. The claim about
hol\-o\-mor\-phic\-i\-ty of Kil\-ling fields on compact
pseu\-do\hs-\hn K\"ahler manifolds appears in a 2012 paper by Ya\-ma\-da, and
in an appendix we provide
a detailed explanation of why we believe that Ya\-ma\-da's argument is
incomplete.
\end{abstract}

\maketitle

\setcounter{section}{0}
\setcounter{theorem}{0}
\renewcommand{\theequation}{\arabic{section}.\arabic{equation}}
\renewcommand{\thetheorem}{\Alph{theorem}}
\section*{Introduction}
\setcounter{equation}{0}
By a pseu\-do\hs-\hn K\"ahler 
manifold we mean a pseu\-do\hs-Riem\-ann\-i\-an manifold 
$\,(M\nh,g)$ endowed with a $\,\nabla\nh$-par\-al\-lel al\-most-com\-plex 
structure $\,J$, for the Le\-vi-Ci\-vi\-ta connection $\,\nabla\hs$ of $\,g$,
such that the operator $\,J\hskip-1.5pt_x\w:\txm\to\txm\,$ is a linear
$\,g_x\w\nh$-isom\-e\-try (or is, equivalently, $\,g_x\w\nh$-skew-ad\-joint) at
every point $\,x\in M\nh$. This implies integrability of $\,J\,$ (see the
comment preceding Lemma 3.1). We then call $\,(M\nh,g)$ a {\it
pseu\-do\hs-\hn K\"ahler\/ $\,\ddb$ manifold\/} if, in addition, the
underlying complex manifold $\,M\,$ has the following $\,\ddb$ {\it property},
also referred to as {\it the\/ $\,\ddb\hs$ lemma\hh}:
\begin{equation}\label{ddb}
\begin{array}{l}
\mathrm{every\ closed\ }\pl\hs\hyp\mathrm{exact\ or\ 
}\,\cro\hs\hyp\mathrm{exact\ }\,(p,q)\,\mathrm{\ form}\\
\mathrm{equals\ }\,\hh\ddb\hh\lambda\,\mathrm{\ for\ some\ 
}\,(p-1,q-1)\,\mathrm{\ form\ }\,\lambda\hh.
\end{array}
\end{equation}
It is well known that the $\,\ddb\,$ property follows if $\,M$ is compact and 
admits a Riemannian K\"ahler metric \cite[Prop.\,6.17 on p.\,144]{voisin}.
\begin{theorem}\label{thrma}Every Kil\-ling vector field on a compact\/
pseu\-do\hs-\hn K\"ahler\/ $\,\ddb$ manifold is real hol\-o\-mor\-phic.
\end{theorem}
We provide two proofs of Theorem~\ref{thrma}, in Sections~\ref{pa}
and~\ref{aa}. The former is derived directly from the $\,\ddb\,$ condition;
the latter, shorter, relies on the Hodge decomposition, which is equivalent to
the $\,\ddb\,$ property 
\cite[p.\,269, subsect.\,(5.21)]{deligne-griffiths-morgan-sullivan}. 

The Riem\-ann\-i\-an-\hn K\"ahler case of Theorem~\ref{thrma} is well known, and
straightforward \cite[the lines following Remark 4.83 on pp.\
60--61]{ballmann}. See also Remark~\ref{riemc}.

For pseu\-do\hs-\hn K\"ahler surfaces, our argument yields a stronger conclusion.
\begin{theorem}\label{thrmb}In real dimension four the assertion of
Theorem\/~{\rm\ref{thrma}} holds without the\/ $\,\ddb$ hypothesis.
\end{theorem}
The authors wish to express their gratitude to Kirollos Masood for bringing 
Ya\-ma\-da's paper \cite{yamada} to the first author's attention and 
discussing with him issues involving Theorem~\ref{thrmb}, formula (\ref{tfr}),
and the Appendix. We
also thank Fangyang Zheng for very useful suggestions about Lemma~\ref{ddbar},
and Takumi Yamada for a brief but helpful communication.

\renewcommand{\thetheorem}{\thesection.\arabic{theorem}}
\section{Proof of Theorem~\ref{thrmb}}\label{pb} 
\setcounter{equation}{0}
All manifolds, mappings, tensor fields and connections are assumed smooth.
\begin{lemma}\label{ldpar}Given a connection\/ $\,\nabla$ on a manifold\/
$\,M\nh$, let a vector field\/ $\,v\,$ on\/ $\,M\,$ be af\-fine in the sense 
that its local flow preserves\/ $\,\nabla\nnh$. Then, for any\/ 
$\,\nabla\nh$-par\-al\-lel tensor field\/ $\,\varTheta\hh$ on\/ $\,M\nh$, of
any type, the Lie derivative\/ $\,\pounds\hskip-1pt_v\w\varTheta\,$ is
$\,\nabla\nh$-par\-al\-lel as well. If\/ $\,\varTheta\,$ happens to be a
closed differential form,
$\,\pounds\hskip-1pt_v\w\varTheta
=\hh d\hs[\varTheta(v,\,\cdot\,,\ldots,\,\cdot\,)]$.
\end{lemma}
\begin{proof}
Clearly, $\,-\nnh\pounds\hskip-1pt_v\w\varTheta\,$ is the derivative with
respect to the real variable $\,t$, at $\,t=0$, of the push-for\-wards
$\,[d\phi\hn_t\w]\varTheta\,$ under the local flow $\,t\mapsto\phi\hn_t\w$ of
$\,v$. All $\,[d\phi\hn_t\w]\varTheta$ being
$\,\nabla\nh$-par\-al\-lel, so is $\,\pounds\hskip-1pt_v\w\varTheta$. For the 
final clause, use Cartan's homotopy formula 
$\,\pounds\hskip-1pt_v\w=\imath_v\w\hn d+d\hs\imath_v\w$ for
$\,\pounds\hskip-1pt_v\w$ acting on differential forms
\cite[Thm.\, 14.35, p.\,372]{lee}.
\end{proof}
Lemma~\ref{ldpar} also follows from the Leib\-niz rule: 
$\,\pounds\hskip-1pt_v\w(\nabla\nh\varTheta)
=(\pounds\hskip-1pt_v\w\nabla)\hs\varTheta
+\nabla\nh(\pounds\hskip-1pt_v\w\varTheta)$.

Let $\,(M\nh,g)\,$ now be a fixed
pseu\-do\hs-\hn K\"ahler manifold. If $\,v\,$ is any
vector field on $\,M\,$ then, with 
$\,J\,$ and $\,\nabla\nh v\,$ treated as bundle mor\-phisms
$\,T\nh M\to T\nh M\nh$, 
\begin{equation}\label{com}
\mathrm{for\ }\,B=\nabla\nh v\,\mathrm{\ and\ }\,A\nh
=\hs\pounds\hskip-1pt_v\w J\,\mathrm{\ one\ has\ }\,A=[J,B\hh]\,\mathrm{\ and\
}\,J\nh A=-\nh AJ\hh,
\end{equation}
which is immediate from the Leib\-niz rule. 
For the K\"ah\-ler form $\,\omega=g(J\hs\cdot\,,\,\cdot\,)\,$ of 
$\,(M\nh,g)\,$ and any $\,g$-\hn Kil\-ling  
vector field $\,v$, it follows from (\ref{com}) and Lemma~\ref{ldpar} that
\begin{equation}\label{pre}
\begin{array}{rl}
\mathrm{i)}&A\,=\hs\pounds\hskip-1pt_v\w J\,\mathrm{\ and\ }\,\,\alpha\,
=\hs\pounds\hskip-1pt_v\w\hs\omega\hs\,\mathrm{\ are\ related\ by\ }\,\alpha
=g(A\hs\cdot\,,\,\cdot\,)\hh,\mathrm{\ while}\\
\mathrm{ii)}&A^*\nh=-\nh A\hh,\hskip7.6ptJ\nh A=-\nh AJ\hh,\hskip7.6pt\nabla\nnh A
=0\hh,\hskip7.6pt\nabla\nh\alpha=0\hh,\mathrm{\ \ and\
}\,\alpha\,\mathrm{\ is\ exact.} 
\end{array}
\end{equation}

Given an exact $\,p\hs$-form $\,\alpha\,$ on a compact
pseu\-do\hs-Riem\-ann\-i\-an manifold $\hs(M\nh,g)$,
\begin{equation}\label{ltw}
\alpha\,\mathrm{ \ is\ }\,L\nh^2\nh\hyp\mathrm{or\-thog\-o\-nal\ to\ all\ 
parallel\ }\,p\,\mathrm{\ times\ co\-var\-i\-ant\ tensor\ fields\
}\,\theta\,\mathrm{\ on\ }\,M.
\end{equation}
Namely, $\,(\theta,\alpha)=(\mu,\alpha)=(\mu,d\hh\beta)
=(d\hh^*\hskip-1.5pt\mu,\beta)\,$ for $\,\beta\,$ with $\,\alpha=d\hh\beta\,$ 
and the skew-sym\-met\-ric part $\,\mu\,$ of $\,\theta$, while
$\,d\hh^*\hskip-1.5pt\mu=0$, as $\,\nabla\nh\mu=0$. Here $\,(\,,\hs)\,$ is the 
$\,L\nh^2$ inner product, assigning to two tensor fields of the same type the
integral over $\,M\,$ of their $\,g$-in\-ner product, and $\,d\hh^*$ denotes
the $\,g$-di\-ver\-gence.
\begin{remark}\label{riemc}By (\ref{pre}\hs-ii) and (\ref{ltw}), for a 
Kil\-ling field $\,v\,$ on a compact 
{\it Riemannian\/} K\"ahler manifold, $\,\pounds\hskip-1pt_v\w\hs\omega\,$ 
is $\,L\nh^2\nh$-or\-thog\-o\-nal to itself, and so, as a consequence of
(\ref{pre}\hs-i), $\,v\,$ must be real hol\-o\-mor\-phic.
\end{remark}
Let $\,(M\nh,g)\,$ be, again, a pseu\-do\hs-\hn K\"ah\-ler manifold. The
vec\-tor bun\-dle
mor\-phisms $\,C:T\nh M\to T\nh M\,$ having $\,C^*=-C\,$ (that is,
$\,g_x\w\nh$-skew-ad\-joint at every point $\,x\in M$) constitute the sections 
of
\begin{equation}\label{vsb}
\mathrm{\ the\ vector\ sub\-bundle\ }\,\mathfrak{so}\hh(T\nh M)\,\mathrm{\
of\ }\,\mathrm{End}_{\hskip.4pt\bbR\hskip0pt}\w(T\nh M)
=\mathrm{Hom}_{\hskip.4pt\bbR\hskip0pt}\w(T\nh M\nh,T\nh M)\hh.
\end{equation}
We denote by $\,\xe\hs$ the vector sub\-bundle of
$\,\mathfrak{so}\hh(T\nh M)$, the sections $\,C\,$ of which are also 
com\-plex-anti\-lin\-e\-ar (so that $\,JC=-C\hn J$, in addition to 
$\,C^*=-C$). Then
\begin{equation}\label{cvb}
\begin{array}{l}
\xe\hs\mathrm{\ is\ a\ complex\ vector\ bundle\ of\ rank\
}\,m(m-1)/2\mathrm{,\ where\ }\,m=\dimc M\nnh,\\
\mathrm{with\ a\ pseu\-do\hs}\hyp\mathrm{Her\-mit\-i\-an\ fibre\ metric\ 
having\ the\ real\ part\ induced\ by\ }g.
\end{array}
\end{equation}
In fact, $\,C\mapsto JC\,$ provides the complex structure for $\,\xe\nh$.
Nondegeneracy of $\,g$ restricted to $\,\xe\hs$ follows from
$\,g$-or\-thog\-o\-nal\-i\-ty of the decomposition
$\,\mathrm{End}_{\hskip.4pt\bbR\hskip0pt}\w(T\nh M)
=\,\mathrm{End}_{\hskip.4pt\bbC\hskip0pt}\w(T\nh M)\oplus\xe\nh\oplus\dz\nh$,
the sections $\,C\,$ of the sub\-bundle $\,\dz\,$ being characterized by
$\,JC=-C\hn J\,$ and $\,C^*\nh=C$, with
$\,\mathrm{End}_{\hskip.4pt\bbC\hskip0pt}\w(T\nh M)\,$ orthogonal to
$\,\xe\hs\oplus\dz\hs$ since any anti\-lin\-e\-ar mor\-phism 
$\,C:T\nh M\to T\nh M\,$ is conjugate, via $\,J$, to $\,-C$, and so 
$\,\mathrm{tr}_{\hskip.4pt\bbR\hskip0pt}\w C=0$. The
pseu\-do\hs-Her\-mit\-i\-an fibre metric in $\,\xe\hs$
arises by restricting
$\,\lg\hh\cdot\hs,\hn\cdot\hh\rg-i\lg J\hs\cdot\hs,\hn\cdot\hh\rg\,$ to
$\,\xe\nh$, for the
pseu\-do\hs-Riem\-ann\-i\-an fibre metric $\,\lg\hh\cdot\hs,\hn\cdot\hh\rg\,$
in $\,\mathrm{End}_{\hskip.4pt\bbR\hskip0pt}\w(T\nh M)\,$ induced by $\,g$.
The rank $\,m(m-1)/2\,$ follows since 
$\,\mathfrak{so}\hh(T\nh M)=\mathfrak{u}\hh(T\nh M)\oplus\xe\nh$, with
$\,\mathfrak{u}\hh(T\nh M)\subseteq\mathfrak{so}\hh(T\nh M)$ characterized
by having sections $\,C:T\nh M\to T\nh M\,$ that commute with $\,J\,$ (which,
due to their $\,g$-skew-ad\-joint\-ness, makes them also 
$\,g\hn^{\mathbf{c}}\nh$-skew-ad\-joint, for $\,g\hn^{\mathbf{c}}\nh
=g-i\hs\omega$): 
$\,\mathfrak{so}\hh(T\nh M)$ and $\,\mathfrak{u}\hh(T\nh M)\,$ have the
real ranks $\,m(2m-1)\,$ and $\,m^2\nh$.
\begin{proof}[Proof of Theorem~\ref{thrmb}]By (\ref{cvb}), with $\,m=2$,
the pseu\-do\hs-Her\-mit\-i\-an fibre metric in the {\it line\/} bundle
$\,\xe\hs$ must be positive or negative definite. Hence so is its
$\,g$-in\-duc\-ed real part. For any Kil\-ling field $\,v$, (\ref{pre}\hs-ii)
implies that $\,A=\pounds\hskip-1pt_v\w J\,$ is a section of $\,\xe\hs$ which,
due to (\ref{pre}) -- (\ref{ltw}), is $\,L\nh^2\nh$-or\-thog\-o\-nal to
itself, and so $\,\pounds\hskip-1pt_v\w J=0$.
\end{proof}
The above proof does not extend to compact pseu\-do\hs-\hn K\"ahler manifolds
$\hs(M\nh,g)\,$ of complex dimensions $\,m>2\,$ with indefinite metrics.
Namely, if
the pair $\,(j,k)$ represents the metric signature of $\,g$, with $\,j\,$
minuses and $\,k\,$ pluses (both $\,j,k$ even, $\,j+k=2m$), then the
analogous signature of the real part (induced by $\,g$) of the
pseu\-\hbox{do\hs-}Her\-mit\-i\-an fibre metric in $\,\xe\hs$ is
$\,(jk/2,\,[\hs j^2\nh+k^2\nh-2(j+k)]/4)$, with both components (indices)
positive unless $\,jk=0\,$ or $\,j=k=2$.

One easily verifies this last claim, about the signature, by using a
$\,J\hskip-1.5pt_x\w$-in\-var\-i\-ant time\-like-space\-like orthogonal
decomposition of $\,\txm\nh$, at any $\,x\in M\nh$, to obtain obvious
three-sum\-mand orthogonal decompositions of both
$\,\mathfrak{so}\hh(T\nh M)\,$ and $\,\mathfrak{u}\hh(T\nh M)$ at $\,x$,
two summands being space\-like, and one time\-like.

\section{Proof of Theorem~\ref{thrma}}\label{pa} 
\setcounter{equation}{0}
We denote by $\,\og^{p,q}\nnh M\,$ the space of com\-plex-val\-ued 
differential $\,(p,q)\,$ forms on a complex manifold $\,M\nh$. On such
$\,M\nh$, as $\,\cro\hs\zeta=0\,$ whenever $\,d\hs\zeta=0$,
\begin{equation}\label{clh}
\mathrm{closedness\ of\ a\ }\,(p,0)\,\mathrm{\ form\ }\,\zeta\,\mathrm{\
implies\ its\ hol\-o\-mor\-phic\-i\-ty.}
\end{equation}
Conversely, according to
\cite[p.\,269, subsect.\,(5.21)]{deligne-griffiths-morgan-sullivan} and
\cite[p.\,101,\,Corollary 9.5]{ueno}, on a compact complex $\,\ddb$ manifold,
\begin{equation}\label{hcl}
\mathrm{all\ hol\-o\-mor\-phic\ differential\ forms\ are\ closed.}
\end{equation}
Since many expositions do not state what happens when, in the $\,\ddb\,$ 
property (\ref{ddb}), $\,p\,$ or $\,q\,$ equals $\,0$, we note that, as 
Fangyang Zheng pointed out to us, (\ref{ddb}) for $\,(p,0)\,$ forms easily
follows from the case where $\,p$ and $\,q\,$ are positive.
\begin{lemma}\label{ddbar}On a compact complex manifold\/ $M\nh$ with the 
``\hh positive\/ $(p,q)$ version'' of the\/ $\,\ddb$ property, if\/
$\,\xi\in\og^{p,0}\nnh M\nh$, for 
$\,p\ge1$, and\/ $\,\pl\hs\xi\,$ is closed, then\/ $\,\pl\hs\xi=0$.  
\end{lemma}
\begin{proof}As $\,0=d\hh\pl\hs\xi=\dbd\hs\xi=-\ddb\hs\xi$, the ``positive''
$\,\ddb\,$ lemma applied to the closed $\,\cro\hs$-exact $\,(p,1)\,$ form
$\,\cro\hs\xi\,$ gives $\,\cro\hs\xi=\dbd\hs\eta\,$ for some
$\,\eta\in\og^{p-1,0}\nnh M\nh$. Being thus hol\-o\-mor\-phic,
$\,\xi-\pl\hs\eta\in\og^{p,0}\nnh M\,$ is closed by (\ref{hcl}), and 
$\,0=\pl\hs(\xi-\pl\hs\eta)=\pl\hs\xi$.
\end{proof}
Lemma~\ref{ddbar} implies, via complex conjugation, its analog 
for $\,(0,q)\,$ forms. Also by Lemma~\ref{ddbar}, 
on a compact complex manifold $\,M\,$ with the $\,\ddb$ property,
\begin{equation}\label{exz}
\mathrm{the\ only\ exact\ }\,(p,0)\,\mathrm{\ form\ }\,\zeta\,\mathrm{\ on\
}\,M\,\mathrm{\ is\ }\,\zeta=0\hh,
\end{equation}
since exactness of $\,\zeta\in\og^{p,0}\nnh M\,$ amounts to its
$\,\pl\hs$-exact\-ness and implies its closedness.

For a pseu\-do\hs-\hn K\"ahler manifold $\,(M\nh,g)$, a
bundle mor\-phism $\,A:T\nh M\to T\nh M\nh$, and the corresponding 
twice-co\-var\-i\-ant tensor field $\,\alpha=g(A\hs\cdot\,,\,\cdot\,)$, one
clearly has
\begin{equation}\label{iff}
\alpha(J\hs\cdot\,,J\hs\cdot\hs)=\pm\hs\alpha\,\mathrm{\ if\ and\ only\ if\
}\,J\nh A=\pm\hh AJ\hh,\quad\mathrm{with\ either\ sign\ }\,\pm\hh.
\end{equation}
Given a pseu\-do\hs-\hn K\"ahler manifold $\,(M\nh,g)$, vector fields $\,u,v\,$ on
$\,M\,$ and sections $\,A,C\,$ of $\,\mathfrak{so}\hh(T\nh M)$, cf.\
(\ref{vsb}), may be used to represent a com\-plex-val\-ued 
$\,1$-form $\,\xi$ and $\,2$-form $\,\zeta\,$ on $\,M\nh$, as follows,
\begin{equation}\label{xie}
\xi=u+i\hh v\hh,\quad\zeta=A+i\hh C\hh,
\end{equation}
meaning that $\,\xi=g(u,\,\cdot\,)+ig(v,\,\cdot\,)\,$ and 
$\,\zeta=g(B\hs\cdot\,,\,\cdot\,)+ig(C\hs\cdot\,,\,\cdot\,)$. 
We prefer not to think of (\ref{xie}) as sections of the complexifications of
$\,T\nh M\,$ or $\,\mathfrak{so}\hh(T\nh M)$. For a vector field $\,v\,$
treated via (\ref{xie}) as a real $\,1$-form, and $\,B=\nabla\nh v$, our
factor convention for the exterior derivative gives
\begin{equation}\label{fcv}
dv\,=\,B\,-\,B^*,\quad\mathrm{\ and\ so\
}\,d\hh(J\hn v)\,=\,\nabla(J\hn v)\,-[\nabla(J\hn v)]^*\hs
=\,JB+B^*\hskip-3ptJ\hh.
\end{equation}
\begin{remark}\label{realp}On a complex manifold, a real-val\-ued $\,2$-form
$\,\alpha\,$ is the real part of a com\-plex-bi\-lin\-e\-ar com\-plex-val\-ued
$\,2$-form $\,\zeta\,$ if and only if 
$\,\alpha(J\hs\cdot\,,J\hs\cdot\hs)=-\alpha$, and then necessarily
$\,\zeta=\alpha-i\hh\alpha(J\hs\cdot\,,\,\cdot\,)$. (This clearly remains
valid for arbitrary twice-co\-var\-i\-ant tensor fields, without
skew-sym\-me\-try.)
\end{remark}
\begin{remark}\label{twozr}For a com\-plex-val\-ued 
$\,2$-form $\,\zeta\,$ on a complex manifold $\,M\nh$, having bi\-de\-gree
$\,(2,0)$, or $\,(0,2)$, or $\,(1,1)\,$ clearly amounts to its being 
com\-plex-bi\-lin\-e\-ar, or bi-anti\-lin\-e\-ar or, respectively,
$\,J$-in\-var\-i\-ant: $\,\zeta(J\hs\cdot\,,J\hs\cdot\hs)=\zeta$. Sums
$\,\zeta\,$ of $\,(2,0)$ and $\,(0,2)\,$ forms are similarly characterized
by $\,J$-anti-in\-var\-i\-ance: $\,\zeta(J\hs\cdot\,,J\hs\cdot\hs)=-\hs\zeta$. 
Thus, by (\ref{iff}), in the pseu\-do\hs-\hn K\"ahler case,
$\,\zeta=A+i\hh C\,$ in (\ref{xie}) is a $\,(1,1)\,$ form if and only if
$\,A\,$ and $\,C\,$ commute with $\,J$.
\end{remark}
\begin{lemma}\label{parxi}For a Kil\-ling vector field\/ $\,v\,$ on a
pseu\-do\hs-\hn K\"ahler manifold\/ $\,(M\nh,g)$, using the notation of\/
{\rm(\ref{xie})}, we have
\begin{equation}\label{dxe}
\begin{array}{l}
\xi\in\og^{1,0}\nnh M\hh,\quad\zeta\in\og^{2,0}\nnh M\hh,\quad\pl\hs\xi\,
=\,\zeta\hh,\quad\cro\hs\xi\,=\,i\hh(J\hn B\hn J-B)\hh,\mathrm{\ \ where}\\
\xi\,=\,J\hn v\,-\,iv\hh,\quad\zeta\,=\,A\,-\,iAJ\hh,\quad\mathrm{with\ \ }\,A
=[J,B\hh]\,\mathrm{\ \ for\ \ }\hs B=\nabla\nh v\hh.
\end{array}
\end{equation}
\end{lemma}
\begin{proof}First, $\,J\hn B\hn J-B$, as well as
$\,A=[J,B\hh]\,$ and $\,AJ$, are $\,g_x\w\nh$-skew-ad\-joint at every point
$\,x\in M\nh$, since so is $\,B=\nabla\nh v$, and $\,A\,$ anti\-com\-mutes
with $\,J$, cf.\ (\ref{com}). Thus, $\,\xi,\zeta\,$ and
$\,\gamma=i\hh(J\hn B\hn J-B)\,$ are indeed  
differential forms of degrees $\,1,\,2,\,2$.

Furthermore, $\,\xi\,$ is com\-plex-lin\-e\-ar, and $\,\zeta\,$
com\-plex-bi\-lin\-e\-ar. This is immediate for $\,\xi$. For $\,\zeta$,
note that 
$\,\zeta=\alpha-i\hh\alpha(J\hs\cdot\,,\,\cdot\,)$, where
$\,\alpha=g(A\hs\cdot\,,\,\cdot\,)$, while (\ref{com}) and (\ref{iff}) give
$\,\alpha(J\hs\cdot\,,J\hs\cdot\hs)=-\alpha$. Now we can use
Remark~\ref{realp}.

Thus, $\,\xi\in\og^{1,0}\nnh M\nh$. Also, according to Remark~\ref{twozr},
$\,\zeta\in\og^{2,0}\nnh M\,$ and $\,\gamma\in\og^{1,1}\nnh M\nh$, since
$\,J\hn B\hn J-B\,$ obviously commutes with $\,J$. 
Finally, for $\,A=[J,B\hh]$, (\ref{fcv}) with $\,B^*\nh=-B\,$ gives 
$\,d\hh\xi=A-2i\hh B=[A-i(J\hn B\hn J+B)]+i\hh(J\hn B\hn J-B)$, while
the summands $\,A-i(J\hn B\hn J+B)=A-iAJ=\hs\zeta\,$ and
$\,i\hh(J\hn B\hn J-B)=\gamma\,$ lie in $\,\og^{2,0}\nnh M$  
and $\,\og^{1,1}\nnh M\nh$, which completes the proof.
\end{proof}
\begin{proof}[Proof of Theorem~\ref{thrma}]By (\ref{pre}) and (\ref{iff}), the
$\,\pl\hs$-exact $\,(2,0)\,$ form $\,\zeta=\pl\hs\xi$ in
(\ref{dxe}) is parallel, and hence closed. Lemma~\ref{ddbar} now gives
$\,\zeta=0$, so that $\,\,\pounds\hskip-1pt_v\w J=A=0\,$ due to 
(\ref{com}) and (\ref{dxe}).
\end{proof}

\section{Another proof of Theorem~\ref{thrma}}\label{aa} 
\setcounter{equation}{0}
On a compact complex manifold $\,M\,$ with the $\,\ddb$ property, every
cohomology space $\,H^k\nnh(M,\bbC)\,$ has the Hodge decomposition
\cite[p.\,269, subsect.\,(5.21)]{deligne-griffiths-morgan-sullivan}:
\begin{equation}\label{hdc}
H^k\nnh(M,\bbC)=H^{k,0}\nnh M\oplus H^{k-1,1}\nnh M\oplus\ldots
\oplus H^{1,k-1}\nnh M\oplus H^{0,k}\nnh M\hh,
\end{equation}
with each $\,H^{p,q}\nnh M\,$
consisting of cohomology classes of closed $\,(p,q)\,$ forms. The complex
conjugation of differential forms descends to a real-lin\-e\-ar involution of 
$\,H^k\nnh(M,\bbC)$, the fixed points of which obviously
are the real cohomology classes (those containing real closed differential
forms). In terms of the decomposition (\ref{hdc}), a complex cohomology class
\begin{equation}\label{crl}
\begin{array}{l}
\mathrm{is\ real\ if\ and\ only\ if,\ for\ all\ }\,p\hh\mathrm{\ and\
}\,q\mathrm{,\nnh\ its\ }\hs H^{q,p}\nh\mathrm{\ com}\hyp\\
\mathrm{ponent\ equals\ the\ conjugate\ of\ its\ }\hs\,H^{p,q}\hn\mathrm{\
component.}
\end{array}
\end{equation}
The standard formula $\,N(\x,\y)=[\x,\y]+J[J\x,\y]+J[\x,J\y]-[J\x,J\y]$, for
the Nijen\-huis tensor $\,N\,$ of an al\-most-com\-plex structure $\,J\,$
on a manifold $\,M\,$ and any vector fields $\,\x,\y$, clearly becomes
\begin{equation}\label{tfr}
N(\x,\y)\,=\,[\nabla_{\!\!J\hn\y}\w J]\x\,-\,[\nabla_{\!\!J\hn\x}\w J]\y\,
+\,J[\nabla_{\nnh\!\x}\w\nnh J]\y\,-\,J[\nabla_{\nh\!\y}\w J]\x
\end{equation}
when one uses any fixed tor\-sion\-free connection $\,\nabla\hs$ on $\,M\nh$.
We call $\,\nabla\hs$ a {\it K\"ah\-ler connection\/} for the given
al\-most-com\-plex structure $\,J\,$ if it is tor\-sion\-free and
$\,\nabla\nnh J=0$. By (\ref{tfr}), $\,J\,$ then must be integrable. This
implies {\it integrability of\/ $\,J\,$ in any \hbox{pseu\-do\hs-} K\"ahler
manifold}, as one then has $\,\nabla\nnh J=0\,$ for the Le\-vi-Ci\-vi\-ta
connection $\,\nabla\nh$.
\begin{lemma}\label{expar}For any $\,\nabla\nh$-par\-al\-lel real\/
$\,2$-form\/ $\,\alpha\,$ on a complex manifold\/ $\,M$ with a K\"ah\-ler
connection\/ $\,\nabla\nnh$, such that\/
$\,\alpha(J\hs\cdot\,,J\hs\cdot\hs)=-\alpha$, the com\-plex-val\-ued\/ 
$\,2$-form $\,\zeta=\alpha-i\alpha(J\hs\cdot\,,\,\cdot\,)$ is
hol\-o\-mor\-phic. If, in addition, $\,M\,$ is also compact and has the\/ 
$\,\ddb$ property, while\/ $\,\alpha\,$ is exact, then\/ $\,\alpha=0$.
\end{lemma}
\begin{proof}The relation $\,\alpha(J\hs\cdot\,,J\hs\cdot\hs)=-\alpha\,$
amounts to com\-plex-bi\-lin\-e\-ar\-i\-ty of $\,\zeta$, and so 
$\,\zeta\in\og^{2,0}\nnh M\,$ (Remarks~\ref{realp} --~\ref{twozr}). Being
$\,\nabla\nh$-par\-al\-lel, $\,\zeta\,$ is closed, and hence 
hol\-o\-mor\-phic due to (\ref{clh}). The final clause: 
exactness of $\,\alpha\,$ makes $\,[i\zeta]\in H^{2,0}\nnh M\,$ 
a real cohomology class, so that, by  (\ref{crl}), $\,\zeta\,$ is exact,
and (\ref{exz}) gives $\,\zeta=0$.
\end{proof}
\begin{proof}[Another proof of Theorem~\ref{thrma}]Given a Kil\-ling field
$\,v$, the differential $\,2$-form $\,\alpha=\pounds\hskip-1pt_v\w\hs\omega\,$ 
is parallel and exact by (\ref{pre}), while (\ref{pre}) gives 
$\,J\nh A=-\nh AJ\,$ for $\,A=\pounds\hskip-1pt_v\w J$, related to
$\,\alpha\,$ via $\,\alpha=g(A\hs\cdot\,,\,\cdot\,)$, and so
$\,\alpha(J\hs\cdot\,,J\hs\cdot\hs)=-\alpha\,$ due to (\ref{iff}). 
Lemma~\ref{expar} and
(\ref{pre}\hs-i) now yield $\,\pounds\hskip-1pt_v\w\hs\omega=\alpha=0\,$ and 
$\,\pounds\hskip-1pt_v\w J=0$.
\end{proof}
We do not know whether -- aside from Theorem~\ref{thrmb} and the Riemannian
case -- Theorem~\ref{thrma} remains valid without the $\hs\ddb\hs$ hypothesis.
For
possible future reference, let us note that, as shown above, one has the
following conclusions about a Kil\-ling field $\,v\,$ on a compact 
pseu\-do\hs-\nh K\"ahler manifold, whether or not the $\,\ddb$ property is
assumed. First, for $\,\alpha=\pounds\hskip-1pt_v\w\hs\omega$, the
com\-plex-val\-ued $\,2$-form
$\,\zeta=\alpha-i\alpha(J\hs\cdot\,,\,\cdot\,)\,$ is parallel and 
hol\-o\-mor\-phic (see the preceding proof and Lemma~\ref{expar}). Also,
by (\ref{pre}), $\,\alpha\,$ is exact, while
$\,A=\pounds\hskip-1pt_v\w J:T\nh M\to T\nh M\,$ is parallel
and com\-plex-anti\-lin\-e\-ar, as well as nil\-po\-tent at every point. This
last conclusion follows since the constant function 
$\,\mathrm{tr}_{\hskip.4pt\bbR\hskip0pt}\w A\hn^k\nh$, with any integer
$\,k\ge1$, has zero integral as a consequence of (\ref{ltw}) applied to
$\,\alpha=g(A\hs\cdot\,,\,\cdot\,)\,$ and
$\,\theta=g(A^{k-1}\cdot\,,\,\cdot\,)$.

\setcounter{section}{1}
\renewcommand{\thesection}{\Alph{section}}
\renewcommand{\theequation}{\Alph{section}.\arabic{equation}}
\setcounter{theorem}{0}
\renewcommand{\thetheorem}{\thesection.\arabic{theorem}}
\section*{Appendix: Ya\-ma\-da's argument}\label{ya}
\setcounter{equation}{0}
Ya\-ma\-da's claim \cite[Proposition 3.1]{yamada} that on a compact
pseu\-do\hs-\hn K\"ahler manifold, Kil\-ling fields are real hol\-o\-mor\-phic, 
has a proof which reads, {\it verbatim},
\begin{equation}\label{vrb}
\begin{array}{l}
\mathrm{Let\ }\,X\,\mathrm{\ be\ a\ Kil\-ling\ vector\ field.\ From\ 
Propositions\ 1.2\ and}\\
\mathrm{2.12,\ }\,\,Z\hs=\hs X\hs-\,\sqrt{-\nnh1\,}JX\,\,\mathrm{\ is\ 
hol\-o\-mor\-phic.\
Because\ the\ real}\\
\mathrm{part\hs\ of\hs\ a\hs\ hol\-o\-mor\-phic\hs\ vector\hs\ field\hs\
is\hs\ an\ infinitesimal\ auto}\hyp\\
\mathrm{mor\-phism\nh\ of\nh\ the\nh\ complex\nh\ structure,\hskip-.2ptwe\nh\
have\nh\ our\nh\ proposition.}
\end{array}
\end{equation}
Proposition 1.2 of \cite{yamada}, cited from Kobayashi's book 
\cite{kobayashi}, 
amounts to the well-known
{\it har\-mon\-ic-flow condition\/} satisfied by Kil\-ling fields $\,v\,$ on 
pseu\-do\hs-Riem\-ann\-i\-an manifolds. 
Thus, 2.12 in (\ref{vrb}) should read 2.14, since Propositions 1.2 and 2.14
refer to the Ric\-ci tensor quite prominently, while 2.12 does not mention 
it at all; also, Proposition 2.14 contains, in its second
part, a hol\-o\-mor\-phic\-i\-ty conclusion.

In the ninth line of the proof of the second part of Proposition 2.14, it is
established -- correctly -- that, for every $\,(1,0)\,$ vector field $\,Y\nh$,
and $\,Z\,$ in (\ref{vrb}), $\,\nabla''\nnh Z$ is
$\,L\nh^2\nh$-or\-thog\-o\-nal to $\,\nabla''Y\nnh$. Then an attempt is
made to conclude that $\,\nabla''\nnh Z=0$, arguing by contradiction: 
if $\,\nabla''\nnh Z\ne0\,$ at some point $\,z_0\w$, one can -- again
correctly -- find $\,Y\hh$ having
$\,g(\nabla''\nnh Z,\nabla''Y)\ne0\,$ everywhere in some neighborhood
of $\,z_0\w$. As a next step, it is claimed that a contradiction arises:
cited {\it verbatim},
\begin{equation}\label{ctr}
\begin{array}{l}
\mathrm{By\ considering\ a\ cut}\hyp\mathrm{off\ function,\ we\ see\ that\
there\ exists}\\
\mathrm{a\ complex\ vector\ field\ }\hs Y\nnh\mathrm{\ such\
that\,}\int\hskip-4.2pt_{_M}\w g(\nabla''\nnh Z,\nnh\nabla''Y)\,dv\ne0.
\end{array}
\end{equation}
It is here that the argument seems incomplete: such a cut-off function
$\,\varphi\,$ equals $\,1\,$ on
some small ``open ball'' $\,B\,$ centered at $\,z_0\w$, and vanishes outside
a larger ``concentric ball'' $\,B'\nh$, and after the original choice
of $\,Y\hh$ has been replaced by $\,\varphi Y\nnh$, there is no way to control
the integral of $\,g(\nabla''\nnh Z,\nabla''\nh(\varphi Y))\,$ over
$\,B'\nh\smallsetminus B\,$ (while the integrals over $\,B\,$ and
$\,M\smallsetminus B'$ have fixed values). More precisely, the sum of the
three integrals must be zero, $\,\nabla''\nnh Z\,$ being 
$\,L\nh^2\nh$-or\-thog\-o\-nal to all $\,\nabla''Y\nnh$.

\end{document}